\documentclass{sinst}

\usepackage{amsbsy,amsfonts,amsmath,amssymb,amsthm}
\usepackage{array}
\usepackage{bm}
\usepackage{color}
\usepackage{enumerate}
\usepackage{mathrsfs}
\usepackage{latexsym}
\usepackage[colorlinks=true]{hyperref}
\usepackage{mathtools}
\hypersetup{urlcolor=blue, citecolor=blue, linkcolor=blue}

\usepackage{hyperref}
\usepackage{mathtools}
\mathtoolsset{showonlyrefs, showmanualtags}
\allowdisplaybreaks
\numberwithin{equation}{section}

\theoremstyle{plain}
\newtheorem{mTh}{Main Theorem} 
\newtheorem{lem}{Lemma}[section]

\theoremstyle{definition}

\def\F{\mathcal{F}}

\def\N{\mathbb{N}}
\def\R{\mathbb{R}}

\def\T{\mathcal{T}}

\def\ds{\displaystyle}

\def\Lap{\mathit{\Delta}}

\DeclareMathOperator{\diver}{div}

\definecolor{wineRed}{rgb}{0.7,0,0.3}
\definecolor{grandBleu}{rgb}{0,0,0.8}
\definecolor{darkGreen}{rgb}{0,0.4,0}
\definecolor{blueViolet}{rgb}{0.4,0,1.0}
\definecolor{bloodOrange}{rgb}{0.85,0.05,0}
\definecolor{mycolor}{rgb}{0.8,0,0.2}

\usepackage[textsize=small]{todonotes}
\usepackage{cite}
\usepackage{comment}

\DeclareMathAlphabet{\mathpzc}{OT1}{pzc}{m}{it}

\usepackage[textsize=small]{todonotes}
\numberwithin{equation}{section}

\begin{document}
\vspace*{-0cm}
\title{Weak solution to {K}{W}{C} systems of pseudo-parabolic type
\vspace{-2ex}
}
\author{Daiki Mizuno
\footnotemark[1]
}
\affiliation{Division of Mathematics and Informatics, \\ Graduate School of Science and Engineering, Chiba University, \\ 1-33, Yayoi-cho, Inage-ku, 263-8522, Chiba, Japan}
\email{d-mizuno@chiba-u.jp}
\vspace{-2ex}

\footcomment{
$^*$\,This work was supported by JST SPRING, Grant Number JPMJSP2109. 
}
\maketitle

\noindent
{\bf Abstract.}
In this paper, a class of systems of pseudo-parabolic PDEs is considered. These systems (S)$_\varepsilon$ are derived as a pseudo-parabolic dissipation system of Kobayashi--Warren--Carter energy, proposed by [Kobayashi et al., Physica D, 140, 141--150 (2000)], to describe planar grain boundary motion. These systems have been studied in [arXiv:2402.10413], and solvability, uniqueness and strong regularity of the solution have been reported under the setting that the initial data is sufficiently smooth. Meanwhile, in this paper, we impose weaker regularity on the initial data, and work on the weak formulation of the systems. In this light, we set our goal of this paper to prove two Main Theorems, concerned with: the existence and the uniqueness of weak solution to (S)$_\varepsilon$, and the continuous dependence with respect to the index $\varepsilon$, initial data and forcings.
\newpage

\section{Introduction}

Let $N \in \{ 1,2,3,4 \}$ be a fixed spatial dimension, and let $\Omega \subset \R^N$ be a bounded domain. When $N \geq 2$, $\Gamma$ denotes a smooth boundary of $\Omega$, and $n_\Gamma$ denotes the outer unit normal on $\Gamma$. Also, let $T > 0$ be a fixed time constant, and let us set $Q := (0,T) \times \Omega$, $\Sigma := (0,T) \times \Gamma$, $H := L^2(\Omega)$, and $V := H^1(\Omega)$. 

In this paper, we consider a class of pseudo-parabolic systems, denoted by (S)$_\varepsilon$ for $\varepsilon \in [0,\infty)$:
\begin{align}
    {\rm (S)}_\varepsilon\qquad&
    \\[-4ex]
    &\left\{ \begin{aligned}
        &\partial_t \eta - \Lap (\eta + \mu^2 \partial_t \eta) + g(\eta) + \alpha'(\eta)\sqrt{\varepsilon^2 + |\nabla \theta|^2} = u, \ \mbox{ in } Q,
        \\
        &\nabla (\eta + \mu^2 \partial_t \eta) \cdot n_\Gamma = 0, \ \mbox{ on } \Sigma,
        \\
        &\eta(0) = \eta_0, \ \mbox{ in } \Omega,
    \end{aligned} \right.
    \\
    &\left\{ \begin{aligned}
        &\alpha_0(\eta) \partial_t \theta - \diver \Bigl( \alpha(\eta) \frac{\nabla \theta}{\sqrt{\varepsilon^2 + |\nabla \theta|^2}} + \nu^2 \nabla \partial_t \theta \Bigr) = v, \ \mbox{ in } Q,
        \\
        &\Bigl( \alpha(\eta) \frac{\nabla \theta}{\sqrt{\varepsilon^2 + |\nabla \theta|^2}} + \nu^2 \nabla \partial_t \theta \Bigr) \cdot n_\Gamma = 0, \ \mbox{ on } \Sigma,
        \\
        &\theta(0) = \theta_0, \ \mbox{ in } \Omega.
    \end{aligned} \right.
\end{align}
The system (S)$_\varepsilon$ is pseudo-parabolic version of KWC-model of planer grain boundary motion, which is proposed by \cite{MR1752970,MR1794359}. In this system, $\eta$ and $\theta$ represent \textit{orientation order} and \textit{orientation angle} in a polycrystal, respectively. $g$ is a perturbation for the orientation order $\eta$, with a non-negative primitive $G$. $\alpha$ and $\alpha_0$ are positive-valued functions, called mobilities of grain boundary motion. $u,v$ are forcing terms. Finally, a pair of functions $[\eta_0, \theta_0]$ is the initial data of $[\eta, \theta]$. 

In this paper, the components of the system (S)$_\varepsilon$ are considered under the following assumptions.

\begin{itemize}
    \item[(A0)] $\mu > 0$, $\nu > 0$ are fixed positive constants.
    \item[(A1)] $g:\R \longrightarrow \R$ is a locally Lipschitz continuous function with a non-negative primitive $G \in C^1(\R)$. In addition, $g$ satisfies the following condition: 
    \begin{equation}
      \varliminf_{\xi \downarrow -\infty} g(\xi) = -\infty \mbox{ and } \varlimsup_{\xi \uparrow \infty} g(\xi) = \infty.
    \end{equation}
    \item[(A2)] $\alpha : \R \longrightarrow [0,\infty)$ is a $C^2$-class convex function, such that $\alpha'(0) = 0$ and $\alpha'' \geq 0$ on $\R$. $\alpha_0: \R \longrightarrow (0,\infty)$ is a locally Lipschitz continuous function. Besides, we suppose $\delta_* := \inf \alpha_0(\R) > 0$.
    \item[(A3)] $u \in L^\infty(Q)$, and $v \in L^2(0,T;H)$.
    \item[(A4)] The initial data $[\eta_0, \theta_0]$ belong to $[V \cap L^\infty(\Omega)] \times V$.
\end{itemize}

The system (S)$_\varepsilon$ is derived from the following energy-dissipation flow:
\begin{gather}
    -\biggl[ \begin{gathered}
        I - \mu^2\Lap_N
        \\
        \alpha_0(\eta(t)) I - \nu^2 \Lap_N
    \end{gathered} \biggr] \biggl[\begin{gathered}
        \partial_t \eta(t)
        \\
        \partial_t \theta(t)
    \end{gathered}\biggr]= \nabla_{[\eta,\theta]}\F_\varepsilon(\eta(t),\theta(t)) + \biggl[ \begin{gathered}
        u(t)
        \\
        v(t)
    \end{gathered} \biggr] \mbox{ in } [H]^2,
    \\
    \mbox{ a.e. } t > 0, \label{EDS}
\end{gather}
where $\Lap_N$ denotes Laplace operator subject to zero-Neumann boundary condition, and $\F_\varepsilon$ is free-energy of grain boundary motion, called \textit{KWC-energy}, and defined as follows:
\begin{gather}
    \F_\varepsilon(\eta,\theta) := \frac{1}{2} \int_\Omega |\nabla \eta|^2\,dx + \int_\Omega G(\eta) \,dx + \int_\Omega \alpha(\eta) \sqrt{\varepsilon^2 + |D\theta|^2} \in [0,\infty].
\end{gather}
In this context, the last term is given by:
\begin{align}
  &\int_\Omega \alpha(\eta) \sqrt{\varepsilon^2 + |D\theta|^2} 
  \\
  &\quad := 
  \inf \Biggl\{ \begin{array}{l|l}
      \ds \varliminf_{n \to \infty} \int_\Omega \alpha(\eta) \sqrt{\varepsilon^2 + |\nabla \varphi|^2} \, dx & 
      \parbox{3.5cm}{
          $ \{ \varphi_n \}_{n = 1}^{\infty} \subset W^{1, 1}(\Omega) $ such that $ \varphi_n \to \theta $ in $ L^1(\Omega) $ as $ n \to \infty $
      }
  \end{array} \Biggr\},
  \\
  &\quad\qquad \mbox{for $ [\eta, \theta] \in H^1(\Omega) \times BV(\Omega) $, and $ \varepsilon \in [0,\infty) $.}
\end{align}
We note that if $\theta \in H^1(\Omega)$, then the integral $\int_\Omega \alpha(\eta) \sqrt{\varepsilon^2 + |D\theta|^2}$ coincides with
\\
$\int_\Omega \alpha(\eta) \sqrt{\varepsilon^2 + |\nabla \theta|^2}\,dx$.

The system \eqref{EDS} can be regarded as a generalized version of KWC system, and the case $\mu = \nu = \varepsilon = 0$ corresponds to the original KWC system. So far, many researchers have worked on mathematical verification of KWC-type systems, and specified that main difficulity of KWC-type systems lies in uniqueness question, because of the the singular flux $\alpha(\eta) \frac{\nabla \theta}{\sqrt{\varepsilon^2 + |\nabla \theta|^2}}$, and the unknown-dependent mobility $\alpha$ and $\alpha_0$. In fact, we need suitable regularization for the energy, and simplification for the mobility to obtain uniqueness result (cf. \cite{MR2469586,MR3155454,MR4395725}).

Recently, however, \cite{antil2024well} reported the solvability, uniqueness, and regularity result for a KWC system of pseudo-parabolic type with sufficiently smooth initial data, such as $[\eta_0, \theta_0] \in [H^2(\Omega) \cap L^\infty(\Omega)] \times H^2(\Omega)$, without any change for the settings of free-energy and unknown-dependent mobilities. Consequently, we can say that pseudo-parabolic regularization is an effective way to address the singularity and complex structure in (S)$_\varepsilon$. In this paper, we focus on unfinished issues on the work \cite{antil2024well}. First, as one of key properties of pseudo-parabolic equations, we cannot expect smoothing effect for initial data in (S)$_\varepsilon$. In view of this, we will deal with setting up the weak formulation of the system under the weak regularity of initial data, such as $[\eta_0, \theta_0] \in [V \times L^\infty(\Omega)] \times V$, and work on verifying the well-posedness of the system. Second, looking towards advancement in optimal control problems, we will provide a result for continuous dependence of solution to (S)$_\varepsilon$, with respect to the change of $\varepsilon \in [0,\infty)$, and with respect to the change of initial data and forcings in suitable topology.

Based on the above background, we set the goal of this article to prove the following two Main Theorems.

\begin{mTh} \label{mth1}
    Under the assumptions (A0)--(A4), the system (S)$_\varepsilon$ admits a unique solution $[\eta, \theta]$ in the following sense:
    \begin{description}
        \item[(S0)] $\eta \in W^{1,2}(0,T;V) \cap L^\infty(Q)$, and $\theta \in W^{1,2}(0,T;V)$. In particular, if $\theta_0 \in L^\infty(\Omega)$ and $v \equiv 0$, then $\theta \in L^\infty(Q)$.
        \item[(S1)] $\eta$ solves the following variational identity:
        \begin{align}
          &(\partial_t \eta(t) + g(\eta(t)) + \alpha'(\eta(t)) \sqrt{\varepsilon^2 + |\nabla \theta(t)|^2}, \varphi)_H 
          \\
          &\quad + (\nabla (\eta + \mu^2 \partial_t \eta)(t), \nabla \varphi)_{[H]^N} = (u(t), \varphi)_H,
          \\
          &\quad\quad \mbox{ for any } \varphi \in V, \mbox{ and a.e. } t \in (0,T).
        \end{align} 
        \item[(S2)] $\theta$ solves the following variational inequality:
        \begin{align}
            &\bigl((\alpha_0(\eta) \partial_t \theta)(t), \theta(t) - \psi\bigr)_H + \int_\Omega \alpha(\eta(t)) \sqrt{\varepsilon^2 + |\nabla \theta(t)|^2}\,dx
            \\
            &\quad + \nu^2 (\nabla \partial_t \theta(t), \nabla (\theta(t) - \psi))_{[H]^N}
            \\
            &\quad \leq \int_\Omega \alpha(\eta(t)) \sqrt{\varepsilon^2 + |\nabla \psi|^2}\,dx + (v(t), \theta(t) - \psi)_H,
            \\
            &\qquad\quad \mbox{ for any } \psi \in V, \mbox{ and a.e. }t \in (0,T).
        \end{align}
        \item[(S3)] $[\eta, \theta]$ fulfills the following energy-inequality:
        \begin{align}
            &C_0 \int_s^t \Bigl( |\partial_t \eta(r)|_V^2 + |\partial_t \theta(r)|_V^2 \Bigr)\, dr + \F_\varepsilon(\eta(t), \theta(t))
            \\
            &\quad \leq \F_\varepsilon(\eta(s), \theta(s)) + \frac{1}{2}\int_s^t \biggl( |u(r)|_{H}^2 + \frac{1}{\delta_*}|v(r)|_{H}^2 \biggr)\,dr, \label{EI}
            \\
            &\qquad\qquad \mbox{ for any } 0 \leq s \leq t \leq T.
        \end{align}
        where $C_0 = \min \{ \frac{1}{4}, \mu^2, \frac{\delta_*}{2}, \nu^2 \}$.
        \item[(S4)] $[\eta(0),\theta(0)] = [\eta_0, \theta_0]$ in $[H]^2$.
      \end{description}
  \end{mTh}
  
  \begin{mTh} \label{mth2}
    Let $\{ \varepsilon_n \}_{n = 1}^\infty \subset [0,\infty)$, $\{ [\eta_{0,n}, \theta_{0,n}] \}_{n=1}^\infty \subset [V \cap L^\infty(\Omega)] \times V$ and $\{ [u_n, v_n] \}_{n=1}^\infty \subset L^\infty(Q) \times L^2(0,T;H)$ be sequences satisfying the following conditions:
    \begin{equation}
      \left\{ \begin{aligned}
        &\bullet \ \sup_{n \in \N} |\eta_{0,n}|_{L^\infty(\Omega)} < \infty, \mbox{ and } \sup_{n \in \N} |u_n|_{L^\infty(Q)} < \infty,
        \\
        &\bullet \ \begin{gathered}
          \varepsilon_n \to \varepsilon, \, \eta_{0,n} \to \eta_0, \, \theta_{0,n} \to \theta_0 \mbox{ in } V, \mbox{ and }
        \\
        u_n \to u, \, v_n \to v \mbox{ weakly in } [L^2(0,T;H)]^2,
        \end{gathered}
      \end{aligned} \right.  \mbox{ as } n \to \infty.\label{CD1}
    \end{equation}
    Let $[\eta, \theta]$ be the unique solution to (S)$_\varepsilon$, corresponding to the initial data $[\eta_0, \theta_0]$ and the forcings $[u,v]$, and let $[\eta_n, \theta_n]$ be the unique solution to (S)$_{\varepsilon_n}$, for the initial data $[\eta_{0,n}, \theta_{0,n}]$ and the forcings $[u_n,v_n]$, for any $n = 1,2,3,\dots$. Then, we can obtain the following convergences as $n \to \infty$:
    \begin{gather}
        \left\{ \begin{aligned}
            &\eta_n \to \eta \mbox{ in } C([0,T];H), \, L^2(0,T;V), \mbox{ weakly in } W^{1,2}(0,T;V)
            \\
            &\qquad\quad \mbox{ and weakly-$*$ in } L^\infty(Q),
            \\
            &\theta_n \to \theta \mbox{ in } C([0,T];H) \mbox{ and weakly in } W^{1,2}(0,T;V),
        \end{aligned} \right. \label{CD99}
    \end{gather}
    and
    \begin{equation}
      \eta_n(t) \to \eta(t), \, \theta_n(t) \to \theta(t) \mbox{ in } V \mbox{ as } n \to \infty, \ \mbox{ for all } t \in [0,T].
    \end{equation}
  \end{mTh}

The outline of this paper is as follows. Notations and key-lemma for the proof are given in the next Section \ref{sec:pre}. On account of these preparations, the proof of the Main Theorems are given in Section \ref{sec:mth1}. 

\section{Preliminaries} \label{sec:pre}
We first prescribe notations and known results used in this paper. 
\medskip

\noindent
\underline{\textbf{\textit{Specific notations.}}}
We define $r \vee s := \max \{ r, s \}$ and $r \wedge s := \min \{r, s\}$ for all $r,s \in [-\infty,\infty]$. We denote by $[\cdot]^+, \, [\cdot]^-$ positive part and negative part, respectively. We denote by $\lfloor \cdot \rfloor$, $\lceil \cdot \rceil$ floor function and ceiling function, respectively. 

Next, to simplify the notations, we set $\{ \gamma_\varepsilon \}_{\varepsilon \in [0,\infty)}$ by letting:
\begin{equation}
  \gamma_\varepsilon:y \in \R^N \mapsto \gamma_\varepsilon(y) := \sqrt{\varepsilon^2 + |y|^2} \in [0,\infty). \label{gamma_epsilon}
\end{equation}
As is easily seen that $\gamma_\varepsilon$ is convex and non-expansive function, and for $0 \leq \alpha^\circ \in L^2(\Omega)$, the following functional is convex and continuous on $[H]^N$:
\begin{equation}
  {\bm w} \in [H]^N \mapsto \int_\Omega \alpha^\circ \gamma_\varepsilon({\bm w})\,dx \in [0,\infty). \label{lsc_onH}
\end{equation}
Also, for any $\varepsilon_0 \in [0,\infty)$, $\gamma_\varepsilon$ converges to $\gamma_{\varepsilon_0}$ uniformly on $\R^N$ as $\varepsilon \to \varepsilon_0$. 

Finally, we use the following fact($*$) in the proof (cf. \cite[Proposition 1.80]{MR1857292}): if $a,b \in \R$ and $\{ a_n \}_{n=1}^\infty, \, \{ b_n \}_{n=1}^\infty \subset \R$ satisfy:
\begin{equation}
  \varliminf_{n \to \infty} a_n \geq a, \ \varliminf_{n \to \infty} b_n \geq b, \mbox{ and } \varlimsup_{n \to \infty} (a_n + b_n) \leq a + b.
\end{equation}
Then, $a_n \to a$ and $b_n \to b$ as $n \to \infty$.

\medskip
\noindent
\underline{\textbf{\textit{Notations for the time-discretization.}}}
Let $\tau > 0$ be a constant of time-step size, and let $\{ t_i \}_{i=0}^\infty \subset [0,\infty)$ be time sequence defined as $t_i = i \tau$, $i = 0,1,2,\dots$. Let $X$ be a Hilbert space. Then, for any sequence $\{ [t_i,z_i] \}_{i=0}^\infty \subset[0,\infty) \times X$, we define three interpolations $\overline{z}_\tau, \, \underline{z}_\tau \in L_{\rm loc}^\infty([0,\infty);X)$ and $z_\tau \in W_{\rm loc}^{1,2}([0,\infty);X)$, by letting:
\begin{gather}
    \overline{z}_\tau(t) = z_i, \ \underline{z}_\tau(t) := z_{i-1}, \ z_\tau(t) := \frac{t - t_{i-1}}{\tau} z_i + \frac{t_i - t}{\tau} z_{i-1}, \label{eq:tI}
    \\
    \mbox{ for } t \in [t_{i-1},t_i), \mbox{ and for } i = 1,2,3,\dots. 
\end{gather}
Here, the following estimates can be obtained by use of Young's and H\"{o}lder's inequality for $t \geq 0$ and $\tau > 0$:
\begin{align}
  &\int_0^t (\partial_t z_\tau(r), \overline{z}_\tau(r))_X \,dr 
  \\
  & \quad \geq \frac{1}{2}\bigl( |\overline{z}_\tau(t)|_X^2 - |z_0|_X^2 \bigr) - \tau^{\frac{1}{2}}|\overline{z}_\tau|_{L^\infty(0,t;X)}|\partial_t z_\tau|_{L^2(0,t+\tau;X)}. \label{tI00}
\end{align}

Meanwhile, for any $ \zeta \in L_\mathrm{loc}^2([0, \infty); X) $, we denote by $ \{ \zeta_i \}_{i = 0}^\infty \subset X $ the sequence of time-discretization data of $ \zeta $, defined as:
\begin{subequations}\label{tI}
\begin{align}\label{tI01}
    & \zeta_0 := 0 \mbox{ in $X$, and } \ \zeta_i := \frac{1}{\tau} \int_{t_{i -1}}^{t_i} \zeta(t) \, dt \  \mbox{ in $ X $, \  for $ i = 1, 2, 3, \dots $.}
\end{align}
Then, the time-interpolations $ \overline{\zeta}_\tau, \underline{\zeta}_\tau $ for the above $ \{ \zeta_i \}_{i = 0}^\infty $ fulfill that:
\begin{align}\label{tI03}
    & \overline{\zeta}_\tau \to \zeta \mbox{ and } \underline{\zeta}_\tau \to \zeta \mbox{ in $ L^2_\mathrm{loc}([0, \infty); X) $, as $ \tau \downarrow 0 $.}
\end{align}
\end{subequations}

\medskip
\noindent
\underline{\textbf{\textit{Time-discretization scheme for (S)$_\varepsilon$.}}}
The proof of Main Theorem 1 is based on time-discretization scheme. In this light, let us recall a known result (cf. \cite[Theorem 1]{antil2024well}).

Let $\tau \in (0,1)$ be a time-step size, and let $\{ t_i \}_{i = 0}^\infty$ be a time-sequence. The index $\varepsilon \in [0,\infty)$ is given arbitrarily. Now, the time-discretization scheme of (S)$_\varepsilon$, denoted by (AP)$_\tau$, is described as follows:

\smallskip
\ (AP)$_\tau$: To find $\{ [\eta_i, \theta_i] \}_{ i=1 }^\infty \subset [V]^2$ satisfying the following variational formulas:
\begin{align}
  &\left\{ \begin{gathered}
    \Bigl( \frac{1}{\tau}(\eta_i - \eta_{i-1}) + g(\T_M \eta_i) + \alpha'(\T_M \eta_i) \gamma_\varepsilon(\nabla \theta_i), \varphi \Bigr)_H + (\nabla \eta_i, \nabla \varphi)_{[H]^N}
    \\
    + \frac{\mu^2}{\tau}(\nabla (\eta_i - \eta_{i-1}) , \nabla \varphi)_{[H]^N} = (u_i, \varphi)_H, \mbox{ for any } \varphi \in V.     
  \end{gathered} \right. \label{AP_eta}
  \\
  &\left\{\begin{aligned}
    &\frac{1}{\tau} (\alpha_0(\T_M \eta_{i-1})(\theta_i - \theta_{i-1}), \theta_i - \psi)_H + \int_\Omega \alpha_M(\eta_{i-1}) \gamma_\varepsilon(\nabla \theta_i)\,dx
    \\
    &\quad + \frac{\nu^2}{\tau} (\nabla (\theta_i - \theta_{i-1}), \nabla (\theta_i - \psi))_{[H]^N}
    \\
    &\quad \leq \int_\Omega \alpha_M(\eta_{i-1}) \gamma_\varepsilon(\nabla \psi)\, dx + (v_i, \theta_i - \psi)_H, \ \mbox{ for any } \psi \in V,
  \end{aligned} \right. \label{AP_theta}
  \\
  &\qquad \mbox{ for } i = 1,2,3, \dots, \mbox{ where } [\eta_0, \theta_0] \mbox{ is the initial data as in (A4)}.
\end{align}
Here, $\T_M: r \in \R \mapsto \T_M(r) := \min \{ M, \max\{ -M,r \} \}$ is a truncation operator with a large constant $M > 0$, fixed later. $\alpha_M$ is a primitive of $\alpha' \circ \T_M$ such that $\alpha_M = \alpha$ on $[-M,M]$. Finally, for $i =1,2,3,\dots$, let $[u_i,v_i] \in [H]^2$ be the time-discretization data of $[u_0^{\rm ex}, v_0^{\rm ex}]$ which is the zero-extension of $[u,v]$.

Based on the above, the following lemma can be obtained through slight modifications of \cite[Theorem 1]{antil2024well}

\begin{lem}\label{AP}
  There exists a sufficiently small constant $\tau_* \in (0,\tau_*)$ such that for any $\tau \in (0,1)$, (AP)$_\tau$ admits a unique solution $\{ [\eta_i, \theta_i] \}_{i = 1}^\infty$, satisfying:
  \begin{align}
    &\frac{C_0}{\tau} \bigl( |\eta_i - \eta_{i-1}|_V^2 + |\theta_i - \theta_{i-1}|_V^2 \bigr) + \F_\varepsilon^M(\eta_i, \theta_i) \label{Energy1}
    \\
    &\quad \leq \F_\varepsilon^M(\eta_{i-1}, \theta_{i-1}) + \frac{\tau}{2}|u_i|_H^2 + \frac{\tau}{2\delta_*}|v_i|_H^2, \ \mbox{ for } i = 1,2,3,\dots,
  \end{align}
  where
  \begin{align}
    \F_\varepsilon^M(\eta,\theta) := \frac{1}{2} \int_\Omega |\nabla \eta|^2\,dx + \int_\Omega G_M(\eta) \,dx + \int_\Omega \alpha_M(\eta) \gamma_\varepsilon(D\theta),
  \end{align}
  and $G_M$ is a non-negative primitive of $g \circ \T_M$.
\end{lem}

\section{Proof of Main Theorems} \label{sec:mth1}
\noindent
\textbf{\boldmath $\S 3.1$ Proof of Main Theorem 1}

Let $\varepsilon \in [0,\infty)$ be fixed. Also, throughout this section, we introduce the following notations $m_\tau(t) := \left\lfloor \textstyle{\frac{t}{\tau}} \right\rfloor, n_\tau(t) := \left\lceil \textstyle{\frac{t}{\tau}} \right\rceil$. In advance of the proof, we prepare two lemmas, concerned with comparison principle for pseudo-parabolic equations. The first one is given in \cite[Lemma 5.4]{antil2024well}.

\begin{lem}\label{CP_eta}
  (cf. \cite[Lemma 5.4]{antil2024well}) Let $\eta^1, \eta^2 \in W^{1,2}(0,T;V)$, $\eta_0^1, \eta_0^2 \in V$, $\widetilde \theta \in L^2(0,T;V)$, $\widetilde u \in L^2(0,T;H)$, and
  \begin{equation}
    \left\{ \begin{aligned}
      &(-1)^{i-1} \left( \partial_t \eta^i - \Lap_N (\eta^i + \mu^2 \partial_t \eta^i ) + g(\T_M \eta^i) + \alpha'(\T_M \eta^i)\gamma_\varepsilon(\nabla \widetilde \theta) \right) 
        \\
        & \qquad \leq (-1)^{i-1} \widetilde u, \mbox{ a.e. in } Q,
      \\
      &\eta^i(0) = \eta_0^i, \mbox{ in } H.
    \end{aligned} \right. \label{CP_eta1}
  \end{equation}
  Then, there exists a constant $C_1 > 0$ such that:
  \begin{equation*}
    \bigl| [\eta^1 - \eta^2]^+(t) \bigr|_V^2 \leq C_1 \bigl| [\eta_0^1 - \eta_0^2]^+ \bigr|_V^2, \ \mbox{ for any } t \in [0,T].
  \end{equation*}
\end{lem}

\begin{lem}\label{CP_theta}
  We assume that $\theta^1, \theta^2 \in W^{1,2}(0,T;V)$, $\theta_0^1, \theta_0^2 \in V$, $\tilde \eta \in W^{1,2}(0,T;H) \cap L^\infty(Q)$, and
  \begin{gather}
    (\alpha_0(\T_M \tilde \eta)\partial_t \theta^i(t), \theta^i(t) - \psi)_H + \int_\Omega \alpha_M(\tilde \eta)\gamma_\varepsilon(\nabla \theta^i(t)) \,dx
    \\
    \qquad + (\nabla \partial_t \theta^i(t), \nabla (\theta^i(t) - \psi))_{[H]^N} \leq \int_\Omega \alpha_M(\tilde \eta) \gamma_\varepsilon(\nabla \psi)\,dx,
     \\
     \mbox{ for any } \psi \in V, \mbox{ and for a.e. } t \in (0,T).\label{CP_theta1} 
  \end{gather}
  Then, there exists a constant $C_2 > 0$ such that:
  \begin{equation}
    \bigl| [\theta^1 - \theta^2]^+(t) \bigr|_V^2 \leq C_2 \bigl| [\theta_0^1 - \theta_0^2]^+ \bigr|_V^2, \ \mbox{ for any } t \in [0,T].
  \end{equation}
\end{lem}

\begin{proof}
  We note that for any $a,b \in \R$, it holds that $a - (a \wedge b) = [a - b]^+ ,  \ \mbox{ and } b - (a \vee b) = - [a-b]^+$. By putting $\psi = (\theta^1 \wedge \theta^2)(t)$ if $i = 1$, and $\psi = (\theta^1 \vee \theta^2)(t)$ if $i = 2$, and taking the sum of two inequalities, we have:
  \begin{gather}
    \bigl( \alpha_0(\tilde \eta(t)) \partial_t (\theta^1 - \theta^2)(t), [\theta^1 - \theta^2]^+(t) \bigr)_H + \frac{\nu^2}{2} \frac{d}{dt} \bigl( |\nabla [\theta^1 - \theta^2]^+(t)|_{[H]^N}^2 \bigr) 
    \\
    \leq 0, \ \mbox{ for a.e. } t \in (0,T),\label{CP_theta2}
  \end{gather}
  Also, in light of (A3), the continuous embedding from $H^1(\Omega)$ to $L^4(\Omega)$ under $N \leq 4$, and the generalized chain rule in BV-theory (cf. \cite[Theorem 3.99]{MR1857292}), it is observed that:
  \begin{align}
    &\quad \bigl( \alpha_0(\tilde \eta(t)) \partial_t (\theta^1 - \theta^2)(t), [\theta^1 - \theta^2]^+(t) \bigr)_H
    \\
    & \geq \frac{1}{2} \frac{d}{dt} \bigl( |\sqrt{\alpha_0(\tilde \eta)} [\theta^1 - \theta^2]^+(t)|_H^2 \bigr)
    \\
    &\qquad - \frac{1}{2}|\alpha_0'(\tilde \eta)|_{L^\infty(Q)} |\partial_t \tilde \eta(t)|_H \bigl| [\theta^1 - \theta^2]^+(t) \bigr|_{L^4(\Omega)}^2 \label{CP_theta3}
    \\
    & \geq \frac{1}{2} \frac{d}{dt} \bigl( |\sqrt{\alpha_0(\tilde \eta)} [\theta^1 - \theta^2]^+(t)|_H^2 \bigr) 
    \\
    &\quad - \frac{(C_V^{L^4})^2}{2} |\alpha_0'(\tilde \eta)|_{L^\infty(Q)} |\partial_t \tilde \eta(t)|_H \bigl| [\theta^1 - \theta^2]^+(t) \bigr|_V^2, \mbox{ for a.e. } t \in (0,T),
  \end{align}
  where $C_{V}^{L^4}$ is the constant of the continuous embedding from $H^1(\Omega)$ to $L^4(\Omega)$. Now, according to \eqref{CP_theta2} and \eqref{CP_theta3}, we can arrive the following Gronwall type inequality:
  \begin{equation}
    \frac{d}{dt}J_0(t) \leq \frac{(C_V^{L^4})^2}{\delta_* \wedge \nu^2}|\alpha_0'(\tilde \eta)|_{L^\infty(Q)}|\partial_t \tilde \eta(t)|_H J_0(t), \mbox{ for a.e. } t \in (0,T), \ \label{CP_theta4}
  \end{equation}
  with
  \begin{equation}
    J_0(t) := \bigl| \sqrt{\alpha_0(\tilde \eta(t))} [\theta^1 - \theta^2]^+(t) \bigr|_H^2 + \nu^2 \bigl| \nabla [\theta^1 - \theta^2](t) \bigr|_{[H]^N}^2,
  \end{equation}
  and thus, we conclude Lemma \ref{CP_theta}.
\end{proof}

\begin{proof}[Proof of Main Theorem \ref{mth1}]
  By (A1), (A3) and (A4), let us set the constant $M$ so large that:
\begin{equation}
  \left\{ \begin{aligned}
    &M \geq \max\{ |\eta_0|_{L^\infty(\Omega)}, |u|_{L^\infty(Q)} \}, \mbox{ and}
    \\
    &g(M) \geq |u|_{L^\infty(Q)}, g(-M) \leq -|u|_{L^\infty(Q)}.
  \end{aligned} \right. \label{Msetting}
\end{equation}

Now, Lemma \ref{AP} yields the following boundedness:
  \begin{itemize}
    \item $\partial_t \eta_\tau, \partial_t \theta_\tau \in L^2(0,\infty;V)$ for $\tau \in (0,\tau_*)$, and 
    \\
    $\sup \{ |\partial_t \eta_\tau|_{L^2(0,\infty;V)} \vee |\partial_t \theta_\tau|_{L^2(0,\infty;V)} \,|\,\tau \in (0,\tau_*) \} < \infty$.
    \item $\{ \eta_\tau \,|\, \tau \in (0,\tau_*) \}$, $\{ \theta_\tau \,|\, \tau \in (0,\tau_*) \}$: bounded in $W^{1,2}(0,T;V)$,
    \item $\{ \overline{\eta}_\tau \,|\, \tau \in (0,\tau_*) \}$, $\{ \underline{\eta}_\tau \,|\, \tau \in (0,\tau_*) \}$, $\{ \overline{\theta}_\tau \,|\, \tau \in (0,\tau_*) \}$, $\{ \underline{\theta}_\tau \,|\, \tau \in (0,\tau_*) \}$: bounded in $L^\infty(0,T;V)$.
  \end{itemize}
  Hence, by applying Aubin's type compactness theory (cf. \cite[Corollary 4]{MR0916688}), we can obtain a sequence $\{ \tau_n \} \subset (0,\tau_*);\, \tau_n \downarrow 0$ and a pair of functions $[\eta, \theta] \in [W^{1,2}(0,T;V)]^2$ such that:
  \begin{gather}
    \eta_n := \eta_{\tau_n} \to \eta, \, \theta_n := \theta_{\tau_n} \to \theta \mbox{ in } C([0,T];H)
    \\
    \mbox{ and weakly in } W^{1,2}(0,T;V), \mbox{ as } n \to \infty.\label{conv_1}
  \end{gather}
  and in particular, 
  \begin{equation}
    [\eta(0), \theta(0)] = \lim_{n \to \infty} [\eta_n(0), \theta_n(0)] = [\eta_0, \theta_0] \mbox{ in } [H]^2. \label{conv_2}
  \end{equation}
  Additionally, taking into account the boundedness of $[\partial_t \eta_n, \partial_t \theta_n]$, we can compute:
  \begin{align}
    &\bigl(|(\overline{\eta}_{\tau_n}- \eta_n)|_{V} \vee |(\underline{\eta}_{\tau_n}- \eta_n)|_{V} \vee |(\overline{\theta}_{\tau_n}- \theta_n)|_{V} \vee |(\underline{\theta}_{\tau_n}- \theta_n)|_{V} \bigr)(t)
    \\
    &\qquad \leq \int_{(t_{i-1},t_i) \cap (0,T)} \bigl( |\partial_t \eta_n(r)|_{V} \vee |\partial_t \theta_n(r)|_V \bigr)\,dr 
    \\
    &\qquad \leq \tau_n^{\frac{1}{2}} \bigl( |\partial_t \eta_n|_{L^2(0,T;V)} \vee |\partial_t \theta_n|_{L^2(0,T;V)} \bigr), \label{dai01}
    \\
    &\quad \mbox{ for } t \in (t_{i-1},t_i) \cap (0,T), \ i = 1,2,3,\dots, n_\tau, \ \mbox{and } \tau \in (0,\tau_*).
  \end{align}
  Hence, we obtain the following convergences:
  \begin{align}
    &\overline{\eta}_n := \overline{\eta}_{\tau_n} \to \eta, \ \underline{\eta}_n := \underline{\eta}_{\tau_n} \to \eta, \mbox{ and } \overline{\theta}_n := \overline{\theta}_{\tau_n} \to \theta, \underline{\theta}_n := \underline{\theta}_{\tau_n} \to \theta,
    \\
    &\qquad \mbox{ in } L^\infty(0,T;H) \mbox{ and weakly-$*$ in } L^\infty(0,T;V), \mbox{ as } n \to \infty, \label{conv_3}
  \end{align}
  and in particular,
  \begin{gather}
    \overline{\eta}_n(t) \to \eta(t), \,\underline{\eta}_n(t) \to \eta(t), \ \overline{\theta}_n(t) \to \theta(t), \,\underline{\theta}_n(t) \to \theta(t)
    \\
    \mbox{ in } H\mbox{ and weakly in } V,  \mbox{ for any } t \in [0,T]. \label{conv_6}
  \end{gather}

  Now, let us show the pair of functions $[\eta, \theta]$ is a solution to the system (S)$_\varepsilon$. The initial condition (S4) is easily confirmed by \eqref{conv_2}. Let us check that $[\eta, \theta]$ solves the variational inequalities (S1) and (S2). From Lemma \ref{AP}, the sequences given in \eqref{conv_1} and \eqref{conv_3} fulfill that for any open interval $I \subset (0,T)$ and $n = 1,2,3,\dots$, 
  \begin{align}
    &\int_I ((\partial_t \eta_n + g(\T_M \overline{\eta}_n) + \alpha'(\T_M \overline{\eta}_n(r))\gamma_\varepsilon(\nabla \overline{\theta}_n(r))), w(r))_H \,dr \label{conv_4}
    \\
    &\quad + \int_I (\nabla (\overline{\eta}_n + \mu^2 \partial_t \eta_n)(r), \nabla w(r))_{[H]^N} \,dr = \int_I (\overline{u}_{\tau_n}(r), w(r))_H \,dr, 
    \\
    &\qquad\qquad\qquad\qquad\qquad \mbox{ for any } w \in L^2(0,T;V), 
  \end{align}
  and,
  \begin{align}
    &\int_I (\alpha_0(\T_M \underline{\eta}_n(r)) \partial_t \theta_n(r), (\overline{\theta}_n -  \omega)(r))_H \,dr
    \\
    &\quad + \int_I \int_\Omega \alpha_M(\underline{\eta}_n(r)) \gamma_\varepsilon(\nabla \overline{\theta}_n(r)) \,dxdr \label{conv_5}
    \\
    &\quad + \nu^2 \int_I (\nabla \partial_t \theta_n(r), \nabla (\overline{\theta}_n - \omega)(r))_{[H]^N} \,dr
    \\
    &\quad \leq \int_I \int_\Omega \alpha_M(\underline{\eta}_n(r)) \gamma_\varepsilon(\nabla \omega(r)) \,dxdr + \int_I (\overline{v}_{\tau_n}(r), (\overline{\theta}_n - \omega)(r))_H \,dr, 
    \\
    &\qquad\qquad\qquad\qquad\qquad \mbox{ for any } \omega \in L^2(0,T;V).
  \end{align}
  Here, let us set $I = (0,t)$ for $t \in (0,T]$, and $w = (\overline{\eta} - \eta)$ in \eqref{conv_4}. Then, using \eqref{tI00}, \eqref{conv_1} and \eqref{conv_3}, it is observed that:
  \begin{align}
    &\varlimsup_{n \to \infty} \Bigl(\int_0^t |\nabla \overline{\eta}_n(r)|_{[H]^N}^2 \,dr + \frac{\mu^2}{2}\bigl( |\nabla \overline{\eta}_n(t)|_{[H]^N}^2 - |\nabla \eta_0|_{[H]^N}^2 \bigr) \Bigr)
    \\
    \leq& \varlimsup_{n \to \infty} \left( \int_0^t |\nabla \overline{\eta}_n(r)|_{[H]^N}^2 \,dr + \mu^2 \int_0^t (\nabla \partial_t \eta_n(r), \nabla \overline{\eta}_n(r))_{[H]^N} \,dr \right) \label{conv_15}
    \\
    \leq& -\lim_{n \to \infty} \int_0^t \bigl( (\partial_t \eta_n + g(\T_M \overline{\eta}_n) + \alpha'(\T_M \overline{\eta}_n)\gamma_\varepsilon(\nabla \overline{\theta}_n))(r), (\overline{\eta}_n - \eta)(r) \bigr)_H\,dr
    \\
    &+ \lim_{n \to \infty} \int_0^t (\nabla (\overline{\eta}_n + \mu^2 \partial_t \eta_n )(r), \nabla \eta(r))_{[H]^N} \,dr
    \\
    &+ \lim_{n \to \infty} \int_0^t (\overline{u}_{\tau_n}(r), (\overline{\eta}_n - \eta)(r))_H \,dr
    \\
    =& \int_0^t |\nabla \eta(r)|_{[H]^N}^2\,dr + \mu^2 \int_0^t (\nabla \partial_t \eta(r), \nabla \eta(r))_{[H]^N} \,dr.
  \end{align}
  In addition, if we take $\omega = \theta$ in \eqref{conv_5}, then having in mind \eqref{tI00}, \eqref{conv_1} and \eqref{conv_3}, we see that:
  \begin{align}
    &\varlimsup_{n \to \infty} \bigl(\int_0^t\int_\Omega \alpha_M(\underline{\eta}_n(r)) \gamma_\varepsilon (\nabla \overline{\theta}_n(r)) \,dxdr + \frac{\nu^2}{2}( |\nabla \overline{\theta}_n(t)|_{[H]^N}^2 - |\nabla \theta_0|_{[H]^N}^2 )\bigr)
    \\
    &\leq\varlimsup_{n \to \infty} \int_0^t\Bigl(\int_\Omega \alpha_M(\underline{\eta}_n(r)) \gamma_\varepsilon (\nabla \overline{\theta}_n(r)) \,dx + \nu^2(\nabla \partial_t \theta_n(r), \nabla \overline{\theta}_n(r))_{[H]^N} \Bigr) \,dr
    \\
    &\leq  \lim_{n \to \infty} \int_0^t (-\alpha_0(\T_M \underline{\eta}_n(r)) \partial_t \theta_n(r) + \overline{v}_{\tau_n}(r), (\overline{\theta}_n - \theta)(r))_H \,dr  \label{conv_10}
    \\
    &\quad + \lim_{n \to \infty} \int_0^t \Bigl( \int_\Omega \alpha_M (\underline{\eta}_n(r)) \gamma_\varepsilon (\nabla \theta(r)) \,dx + \nu^2 (\nabla \partial_t \theta_n(r), \nabla \theta (r))_{[H]^N}\Bigr) dr
    \\
    &= \int_0^t \int_\Omega \alpha_M (\eta(r)) \gamma_\varepsilon(\nabla \theta(r)) \,dxdr + \nu^2 \int_0^t (\nabla \partial_t \theta(r), \nabla \theta(r))_{[H]^N} \,dr.
  \end{align}\noeqref{conv_10}
  On the other hand, from \eqref{lsc_onH}, \eqref{conv_3}, \eqref{conv_6} and Fatou's lemma, we have:
  \begin{align}
    &\varliminf_{n \to \infty} \int_0^t |\nabla \overline{\eta}_n(r)|_{[H]^N}^2 \,dr \geq \int_0^t |\nabla \eta(r)|_{[H]^N}^2 \,dr, \mbox{ and } \label{conv_16}
    \\
    &\varliminf_{n \to \infty} \int_0^t \int_\Omega \alpha_M (\underline{\eta}_n(r)) \gamma_\varepsilon(\nabla \overline{\theta}_n(r)) \,dxdr \geq \int_0^t \int_\Omega \alpha_M (\eta(r)) \gamma_\varepsilon(\nabla \theta(r)) \,dxdr.
  \end{align}\noeqref{conv_16}
  Also, by \eqref{tI00} and \eqref{conv_6}, we see the following items:
  \begin{align}
    &\varliminf_{n \to \infty} \int_0^t (\nabla \partial_t \eta_n(r), \nabla \overline{\eta}_n(r))_{[H]^N} \,dr
    \\
    &\qquad \geq \frac{1}{2}\bigl( |\nabla \eta(t)|_{[H]^N}^2 - |\nabla \eta_0|_{[H]^N}^2 \bigr) = \int_0^t (\nabla \partial_t \eta(r), \nabla \eta(r))_{[H]^N} \,dr,
    \\
    &\varliminf_{n \to \infty} \int_0^t (\nabla \partial_t \theta_n(r), \nabla \overline{\theta}_n(r))_{[H]^N} \,dr \label{conv_7}
    \\
    &\qquad \geq \frac{1}{2}\bigl( |\nabla \theta(t)|_{[H]^N}^2 - |\nabla \theta_0|_{[H]^N}^2 \bigr) = \int_0^t (\nabla \partial_t \theta(r), \nabla \theta(r))_{[H]^N} \,dr. 
  \end{align}
  
  Now, by the fact($*$) in Section \ref{sec:pre}, \eqref{conv_6}, \eqref{conv_15}--\eqref{conv_7}, and the uniform convexity of $[H]^N$, we can derive the following convergences as $n \to \infty$:
  \begin{align}
    &\bullet \, \begin{aligned}
      &\int_0^t |\nabla \overline{\eta}_n (r)|_{[H]^N}^2 \,dr \to \int_0^t |\nabla \eta (r)|_{[H]^N}^2 \,dr,
      \\
      &\quad\mbox{ and therefore, } \overline{\eta}_n \to \eta \mbox{ in } L^2(0,T;V),
    \end{aligned}
    \\
    &\bullet \, \int_0^t (\nabla \partial_t \eta_n(r), \nabla \overline{\eta}_n(r))_{[H]^N}^2 \,dr \to \int_0^t (\nabla \partial_t \eta(r), \nabla \eta(r))_{[H]^N}^2 \,dr, \label{conv_17}
    \\
    &\bullet \, \int_0^t (\nabla \partial_t \theta_n(r), \nabla \overline{\theta}_n(r))_{[H]^N} \,dr \to \int_0^t (\nabla \partial_t \theta(r), \nabla \theta(r))_{[H]^N} \,dr
    \\
    &\bullet \, |\nabla \overline{\eta}_n(t)|_{[H]^N} \to |\nabla \eta(t)|_{[H]^N}, \mbox{ and therefore, } \overline{\eta}_n(t) \to \eta(t) \mbox{ in } V,
    \\
    &\bullet \, |\nabla \overline{\theta}_n(t)|_{[H]^N} \to |\nabla \theta(t)|_{[H]^N}, \mbox{ and therefore, } \overline{\theta}_n(t) \to \theta(t) \mbox{ in } V.
  \end{align}
  Moreover, with \eqref{dai01} and \eqref{conv_17} in mind, we find the following convergences for $t \in [0,T]$:
  \begin{equation}
    \eta_n (t),\underline{\eta}_n(t) \to \eta(t), \ \theta_n(t), \underline{\theta}_n(t) \to \theta(t) \mbox{ in } V \mbox{ as } n \to \infty. \label{conv_18}
  \end{equation}
  
  Then, in view of \eqref{conv_1}, \eqref{conv_3} and \eqref{conv_17}, when $\omega = \psi$ in $V$ in \eqref{conv_5}, we let $n \to \infty$ and observe that for any open interval $I \subset (0,T)$:
  \begin{align}
    &\int_I (\alpha_0(\T_M \eta(r)) \partial_t \theta(r), \theta(r) - \psi)_H \,dr + \int_I \int_\Omega \alpha_M(\eta(r)) \gamma_\varepsilon(\nabla \theta(r)) \,dxdr 
    \\
    &\quad + \nu^2 \int_I (\nabla \partial_t \theta(r), \nabla (\theta(r) - \psi))_{[H]^N} \,dr \label{conv_19}
    \\
    &\quad \leq \int_I \int_\Omega \alpha_M(\eta(r)) \gamma_\varepsilon(\nabla \psi) \,dxdr + \int_I (v(r), \theta(r) - \psi)_H \,dr.
  \end{align}
  Also, using \eqref{conv_1}, \eqref{conv_17}, continuous embedding from $H^1(\Omega)$ to $L^4(\Omega)$, and Lebesgue's dominated convergence theorem, putting $w = \varphi$ in $V$ and letting $n \to \infty$ yield that for any open interval $I \subset (0,T)$:
  \begin{gather}
    \int_I \bigl((\partial_t \eta + (g(\T_M \eta)))(r), \varphi\bigr)_H \,dr + \int_I (\nabla (\eta + \mu^2 \partial_t \eta)(r), \nabla \varphi)_{[H]^N} \,dr
    \\
    + \int_I \int_\Omega \alpha'(\T_M \eta(r)) \varphi \gamma_\varepsilon(\nabla \theta(r)) \,dxdr = \int_I (u(r), \varphi)_H \,dr.\label{conv_20}
  \end{gather}
  Since the interval $I \subset (0,T)$ is arbitrary, the limiting pair $[\eta, \theta]$ fulfills (S1) and (S2) if $|\eta|_{L^\infty(Q)} \leq M$. 
  
  Here, let us confirm the $L^\infty$-boundedness for the limiting function $\eta$, and for $\theta$ when $v \equiv 0$. Take into account (A2) and \eqref{Msetting}, it follows that:
  \begin{gather}
    \left\{ \begin{aligned}
      &\partial_t M - \Lap (M + \mu^2 \partial_t M) + g(M) + \alpha'(M)|\nabla \theta(t)| \geq u(t),
      \\
      &\partial_t (-M) - \Lap (-M + \mu^2 \partial_t (-M)) + g(-M) + \alpha'(-M)|\nabla \theta(t)| \leq u(t),
    \end{aligned} \right. \label{CP_1}
    \\
    \mbox{a.e. on } \Omega, \ \mbox{and for a.e. } t \in (0,T).
  \end{gather}
  Hence, applying Lemma \ref{CP_eta} with:
  \begin{equation}
    \biggl\{ \begin{aligned}
      &[\eta^1, \eta^2, \tilde \theta, \tilde u] = [\eta, M, \theta, u] \mbox{ in } L^2(0,T;H), \, [\eta_0^1, \eta_0^2] = [\eta_0, M] \mbox{ in } H,
      \\
      &[\eta^1, \eta^2, \tilde \theta, \tilde u] = [-M, \eta, \theta, u] \mbox{ in } L^2(0,T;H), \, [\eta_0^1, \eta_0^2] = [-M,\eta_0] \mbox{ in } H,
    \end{aligned} \biggr.
  \end{equation}
  one can see that $|\eta(t)|_{L^\infty(Q)} \leq M$. Also, since any constant function satisfies the variational inequality \eqref{CP_theta1}, if we suppose $\theta_0 \in L^\infty(\Omega)$ and $v \equiv 0$, then applying Lemma \ref{CP_theta} under:
  \begin{equation}
    \left\{\begin{aligned}
      &[\theta^1,\theta^2,\tilde \eta] = [\theta,|\theta_0|_{L^\infty(\Omega)},\eta] \mbox{ in } L^2(0,T;H),
      \\
      &[\theta_0^1, \theta_0^2] = [\theta_0, |\theta_0|_{L^\infty(\Omega)}] \mbox{ in } H,
      \\
      &[\theta^1,\theta^2,\tilde \eta] = [-|\theta_0|_{L^\infty(\Omega)},\theta,\eta] \mbox{ in } L^2(0,T;H),
      \\
      &[\theta_0^1, \theta_0^2] = [-|\theta_0|_{L^\infty(\Omega)},\theta_0] \mbox{ in } H,
    \end{aligned}\right.
  \end{equation}
  we obtain that $|\theta|_{L^\infty(Q)} \leq |\theta_0|_{L^\infty(\Omega)}$. Therefore, $[\eta, \theta]$ fulfills (S0)--(S2).
  
  Next, let us proceed to verify the energy inequality \eqref{EI}. Let us fix $s,t \in [0,T];\, s < t$. For any $n = 1,2,3,\dots$, by summing up the both side of \eqref{Energy1} for $i = n_{\tau_n}(s), n_{\tau_n}(s) + 1, \dots, n_{\tau_n}(t)$, it is observed that:
  \begin{align}
    &\quad C_0 \int_s^t \Bigl( |\partial_t \eta_n(r)|_V^2 + |\partial_t \theta_n(r)|_V^2 \Bigr)\,dr + \F_\varepsilon^M(\overline{\eta}_n(t), \overline{\theta}_n(t)) \label{Energy8}
    \\
    &\leq C_0 \int_{m_{\tau_n}(s) \tau_n}^{n_{\tau_n}(t) \tau_n} \Bigl( |\partial_t \eta_n(r)|_V^2 + |\partial_t \theta_n(r)|_V^2 \Bigr)\,dr + \F_\varepsilon^M(\overline{\eta}_n(t), \overline{\theta}_n(t))
    \\
    &\leq \F_\varepsilon^M(\underline{\eta}_n(s), \underline{\theta}_n(s)) + \frac{1}{2}\int_{m_{\tau_n}(s) \tau_n}^{n_{\tau_n}(t) \tau_n} \left( |u_0^{\rm ex}(r)|_H^2 + \frac{1}{\delta_*}|v_0^{\rm ex}(r)|_H^2  \right)\,dr.
  \end{align}
  On this basis, owing to the convergences \eqref{conv_1}, \eqref{conv_17}, \eqref{conv_18}, the $L^\infty$-boundedness of $\eta$, the estimate \eqref{Energy8}, and Lebesgue's dominated convergence theorem, letting $n \to \infty$ yields that:
  \begin{align}
    &\quad C_0 \int_s^t \Bigl( |\partial_t \eta(r)|_V^2 + |\nabla \partial_t \theta(r)|_V^2 \Bigr)\,dr + \F_\varepsilon(\eta(t), \theta(t))
    \\
    &\leq \varliminf_{n \to \infty} C_0 \int_s^t \Bigl( |\partial_t \eta_n(r)|_H^2 + |\nabla \partial_t \theta_n(r)|_V^2 \Bigr)\,dr + \lim_{n \to \infty} \F_\varepsilon^M(\overline{\eta}_n(t), \overline{\theta}_n(t)) \label{Energy9}
    \\
    &\leq \lim_{n \to \infty} \F_\varepsilon^M(\underline{\eta}_n(s), \underline{\theta}_n(s)) + \frac{1}{2} \lim_{n \to \infty} \int_{m_{\tau_n}^s \tau_n}^{n_{\tau_n}^t \tau_n} \Bigl( |u_0^{\rm ex}(r)|_H^2 + \frac{1}{\delta_*}|v_0^{\rm ex}(r)|_H^2  \Bigr)\,dr
    \\
    &= \F_\varepsilon(\eta(s), \theta(s)) + \frac{1}{2}\int_s^t \Bigl(|u(r)|_H^2 + \frac{1}{\delta_*}|v(r)|_H^2 \Bigr) \, dr.
  \end{align}
and hence the solution $[\eta,\theta]$ fulfills the energy inequality.

Now, our remaining task is the verification of uniqueness of solution. The proof of uniqueness is given in \cite[Main Theorem 2]{antil2024well}. So, In this article, we introduce the outline of proof. Let $[\eta^k, \theta^k], \, k = 1,2$, be the solutions to (S)$_\varepsilon$ corresponding to the same initial value $[\eta_0, \theta_0]$ and forcings $[u,v]$. Let us set $M_0 := |\eta^1|_{L^\infty(Q)} \vee |\eta^2|_{L^\infty(Q)}$, take the difference between the variational identities for $\eta^k, \, k = 1,2$, and put $\varphi := (\eta^1 - \eta^2)(t)$. Then, by using (A1), (A2), convexity of $\alpha$ and H\"{o}lder's and Young's inequality, we can compute as follows:
\begin{align}
  &\frac{1}{2}\frac{d}{dt}\bigl( |(\eta^1 - \eta^2)(t)|_H^2 + \mu^2 |\nabla (\eta^1 - \eta^2)(t)|_{[H]^N}^2 \bigr) \label{uni_1}
  \\
  &\quad \leq |g'|_{L^\infty(-M_0,M_0)}|(\eta^1 - \eta^2)(t)|_H^2 
  \\
  &\qquad + \frac{|\alpha'|_{L^\infty(-M_0,M_0)}}{2} \bigl( |\eta^1 - \eta^2(t)|_H^2 + |\nabla (\theta^1 - \theta^2)(t)|_{[H]^N}^2 \bigr),
  \\
  &\qquad\qquad\qquad\qquad \mbox{ for a.e. } t \in (0,T).
\end{align}

On the other hand, we consider putting $\psi = \theta^2$ in the variational inequality for $\theta^1$, and $\psi = \theta^1$ in the one for $\theta^2$, and adding the both sides of two inequalities. Then, using (A2), continuous embedding from $V$ to $L^4(\Omega)$, and generalized chain rule in BV-theory, we can compute as follows:
\begin{align}
  &\quad \frac{1}{2}\frac{d}{dt} \bigl( |\alpha_0(\eta^1)(\theta^1 -\theta^2)(t)|_H^2 + \nu^2 |\nabla (\theta^1 - \theta^2)(t)|_{[H]^N}^2 \bigr)
  \\
  &\leq \frac{|\alpha'|_{L^\infty(-M_0, M_0)}}{2}\bigl( |(\eta^1 - \eta^2)(t)|_H^2 + |\nabla (\theta^1 - \theta^2)(t)|_{[H]^N}^2\bigr)
  \\
  &\qquad + \frac{\bigl(C_V^{L^4}\bigr)^2|\alpha_0'|_{L^\infty(-M_0,M_0)}}{2} (|\partial_t \eta^1(t)|_H + |\partial_t \theta^2(t)|_V) \cdot \label{uni_3}
  \\
  &\qquad\qquad \cdot \bigl( |(\eta^1 - \eta^2)(t)|_H^2 + |(\theta^1 - \theta^2)(t)|_V^2 \bigr), \mbox{ for a.e. } t \in (0,T).
\end{align}

Therefore, putting
\begin{align}
  J(t) &:= |(\eta^1 - \eta^2)(t)|_H^2 + \mu^2 |\nabla (\eta^1 - \eta^2)(t)|_{[H]^N}^2 
  \\
  &\quad + |\sqrt{\alpha_0(\eta^1(t))}(\theta^1 - \theta^2)(t)|_H^2 + \nu^2 |\nabla (\theta^1 - \theta^2)(t)|_{[H]^N}^2, \ \mbox{ for } t \geq 0, 
  \\
  C_3 &:= \frac{2\bigl( |\alpha'|_{L^\infty(-M_0,M_0)} + |g'|_{L^\infty(-M_0,M_0)} + (C_V^{L^4})^2 |\alpha_0'|_{L^\infty(-M_0,M_0)}\bigr)}{1 \wedge \delta_* \wedge \nu^2},
\end{align}
it is deduced from \eqref{uni_1} and \eqref{uni_3} that:
\begin{equation}
  \frac{d}{dt} J(t) \leq C_3 \bigl( |\partial_t \eta^1(t)|_H + |\partial_t \theta^2(t)|_V + 1 \bigr) J(t) , \ \mbox{ a.e. } t > 0, \label{uni_5}
\end{equation}
\eqref{uni_5} implies that the solution to (S)$_\varepsilon$ is unique, and thus, we complete the proof of Main Theorem \ref{mth1}.
\end{proof}

\smallskip
\noindent
\textbf{\boldmath $\S$ 3.2 Proof of Main Theorem 2}

By setting a large constant $M^* > 0$ such that:
  \begin{gather}
    {M^* \geq \sup_{n \in \N} (|\eta_{0,n}|_{L^\infty(\Omega)} \vee |u_n|_{L^\infty(Q)}), \mbox{ and }}
    \\
    {g(M^*) \geq \sup_{n \in \N} |u_n|_{L^\infty(Q)}, \ g(-M^*) \leq \sup_{n \in \N} -|u_n|_{L^\infty(Q)}}.
  \end{gather}
  and applying the same argument in Main Theorem 1, we obtain $\sup_{n \in \N} |\eta_n|_{L^\infty(Q)} \leq M^*$. Additionally, thanks to Main Theorem 1, the sequence $\{ [\eta_n, \theta_n] \}_{n=1}^\infty$ satisfies the following energy-inequality for $n = 1,2,3,\dots$:
  \begin{align}
    &C_0 \int_0^T \bigl( |\partial_t \eta_n(r)|_V^2 + |\partial_t \theta_n(r)|_V^2 \bigr)\, dr + \F_{\varepsilon_n}(\eta_n(T), \theta_n(T))
    \\
    &\quad \leq \F_{\varepsilon_n}(\eta_{0,n}, \theta_{0,n}) + \frac{1}{2}\int_0^T \bigl( |u_n(r)|_{H}^2 + \frac{1}{\delta_*}|v_n(r)|_{H}^2 \bigr)\,dr. \label{CD3}
  \end{align}

  Now, using the assumption \eqref{CD1}, energy-inequality \eqref{CD3} and $L^\infty$-boundedness of $\eta_n$, we can derive the following boundedness:
  \begin{itemize}
    \item $\{ \eta_n \}_{n = 1}^\infty$: bounded in $W^{1,2}(0,T;V)$ and in $L^\infty(Q)$,
    \item $\{ \theta_n \}_{n = 1}^\infty$: bounded in $W^{1,2}(0,T;V)$,
  \end{itemize}
  Using Aubin's type compactness theory, we can find a subsequence $\{ n_k \} \subset \{ n \} $; $n_k \uparrow \infty$ as $k \to \infty$, and a pair of functions $[\bar \eta, \bar \theta] \in [W^{1,2}(0,T;V) \cap L^\infty(Q)] \times W^{1,2}(0,T;V)$ such that
  \begin{gather}
    \left\{ \begin{aligned}
      &\eta_{n_k} \to \bar \eta \mbox{ in } C([0,T];H), \mbox{ weakly in } W^{1,2}(0,T;V),
      \\
      &\quad \mbox{ and weakly-$*$ in } L^\infty(Q),
      \\
      &\theta_{n_k} \to \bar \theta \mbox{ in } C([0,T];H), \mbox{ weakly in } W^{1,2}(0,T;V),
    \end{aligned} \right. \mbox{ as } k \to \infty. \label{CD6}
  \end{gather}
  In particular, by \eqref{CD1}, we see that:
  \begin{equation}
    [\bar \eta(0), \bar \theta(0)] = \lim_{k \to \infty} [\eta_{n_k}(0), \theta_{n_k}(0)] = \lim_{k \to \infty} [\eta_{0,n_k}, \theta_{0,n_k}] = [\eta_0, \theta_0] \mbox{ in } [H]^2. \label{CD7}
  \end{equation}
  Besides, we have the following convergence:
  \begin{gather}
    \eta_{n_k}(t) \to \bar \eta(t), \theta_{n_k}(t) \to \bar \theta(t) \mbox{ in } H \mbox{ and weakly in } V, \ \mbox{for any } t \in [0,T]. \label{CD8}
  \end{gather}

  Now, we derive additional convergence of $\eta_{n_k}$ and $\theta_{n_k}$, and confirm that the pair $[\bar \eta, \bar \theta]$ coincides the (unique) solution to (S)$_\varepsilon$. First, for any $k = 1,2,3,\dots$, let us set $\varphi = (\eta_{n_k} - \bar \eta)(t)$ in (S1), $\psi = \bar \theta$ in (S2), and integrate the both side over $(0,t)$. 
  Then, applying similar arguments in \eqref{conv_15}--\eqref{conv_17}, we can derive additional convergences as $k \to \infty$:
  \begin{align}
    &\bullet \, \eta_{n_k} \to \bar \eta \mbox{ in } L^2(0,T;V),
    \\
    &\bullet \, \int_0^t (\nabla \partial_t \eta_{n_k}(r), \nabla \eta_{n_k})_{[H]^N} \,dr \to \int_0^t (\nabla \partial_t \bar \eta(r), \nabla \bar \eta)_{[H]^N} \,dr, \label{dai04}
    \\
    &\bullet \, \int_0^t (\nabla \partial_t \theta_{n_k}(r), \nabla \theta_{n_k})_{[H]^N} \,dr \to \int_0^t (\nabla \partial_t \bar \theta(r), \nabla \bar \theta)_{[H]^N} \,dr,
    \\
    &\bullet \, \eta_{n_k}(t) \to \bar \eta(t), \, \theta_{n_k}(t) \to \bar \theta(t) \mbox{ in } V.
  \end{align}
  
  Here, the pairs $[\eta_{n_k}, \theta_{n_k}]$ satisfy the following inequalities:
  \begin{gather}
    \int_I ((\partial_t \eta_{n_k} + g(\eta_{n_k}))(t), \varphi)_H \,dt + \int_I (\nabla (\eta_{n_k} + \mu^2 \partial_t \eta_{n_k})(t), \nabla \varphi)_{[H]^N}\,dt 
    \\
    + \int_I \int_\Omega \alpha'(\eta_{n_k}(t)) \varphi \, \gamma_{\varepsilon_{n_k}}(\nabla \theta_{n_k}(t))\, dxdt = \int_I (u_{n_k}(t), \varphi)_H \,dt, \label{dai02}
    \\
    \mbox{ for any } \varphi \in V,
  \end{gather}
  and
  \begin{align}
    &\int_I (\alpha_0(\eta_{n_k}(t)) \partial_t \theta_{n_k}(t), \theta_{n_k}(t) - \psi)_H \,dt 
    \\
    &\quad + \int_I \int_\Omega \alpha(\eta_{n_k}(t)) \gamma_{\varepsilon_{n_k}}(\nabla \theta_{n_k}(t)) \,dxdt \label{dai03}
    \\
    &\quad + \nu^2 \int_I (\nabla \partial_t \theta_{n_k}(t), \nabla (\theta_{n_k}(t) - \psi))_{[H]^N} \,dt 
    \\
    &\quad \leq \int_I \int_\Omega \alpha(\eta_{n_k}(t)) \gamma_{\varepsilon_{n_k}}(\nabla \psi) \,dxdt + \int_I (v_{n_k}(t), \theta_{n_k}(t)- \psi)_H \,dt
    \\
    &\qquad\qquad\qquad\qquad \mbox{ for any } \psi \in V.
  \end{align}
  Then, using \eqref{dai04} and uniformly convergence of $\gamma_{\varepsilon_{n_k}}$ on $\R^N$, letting $k \to \infty$ in \eqref{dai03} yields that:
  \begin{align}
    &\int_I (\alpha_0(\bar \eta(t)) \partial_t \bar \theta(t), \bar \theta(t) - \psi)_H \,dt + \nu^2 \int_I (\nabla \partial_t \bar \theta(t), \nabla (\bar \theta(t) - \psi))_{[H]^N} \,dt 
    \\
    &\quad + \int_I \int_\Omega \alpha(\bar \eta(t)) \gamma_{\varepsilon}(\nabla \bar \theta(t)) \,dxdt \label{CD13}
    \\
    &\quad \leq \int_I \int_\Omega \alpha(\bar \eta(t)) \gamma_{\varepsilon}(\nabla \psi) \,dxdt + \int_I (v(t), \bar \theta(t)- \psi)_H \,dt.
  \end{align}
  
  On the other hand, using \eqref{CD6}, \eqref{dai04}, continuous embedding from $H^1(\Omega)$ to $L^4(\Omega)$ and Lebesgue's dominated convergence theorem, we obtain that as $k \to \infty$ in \eqref{dai02}:
  \begin{gather}
    \int_I (\partial_t \bar \eta(t) + g(\bar \eta(t)), \varphi)_H \,dt + \int_I (\nabla (\bar \eta + \mu^2 \partial_t \bar \eta)(t), \nabla \varphi)_H\,dt \label{CD14}
    \\
    + \int_I \int_\Omega \alpha'(\bar \eta(t)) \varphi \, \gamma_{\varepsilon}(\nabla \bar \theta(t))\, dxdt = \int_I (u(t), \varphi)_H \,dt.
  \end{gather}
  From \eqref{CD6}, \eqref{CD7}, \eqref{CD13} and \eqref{CD14}, the pair $[\bar \eta, \bar \theta]$ is a solution to (S)$_\varepsilon$, and therefore, we see that $[\bar \eta, \bar \theta]$ coincides the unique solution $[\eta, \theta]$.

  Finally, by the uniqueness of the limit, the convergence \eqref{CD99} is verified. Thus, we complete the proof of Main Theorem \ref{mth2}.

\end{document}